\documentclass[12pt]{elsarticle}
\usepackage{natbib}

\usepackage{amsthm,amsmath,amssymb}

\usepackage{graphicx}
\usepackage{caption}

\usepackage[colorlinks=true,citecolor=black,linkcolor=black,urlcolor=blue]{hyperref}
\usepackage{tikz}
\usetikzlibrary{arrows,positioning,,decorations.pathreplacing}
\usepackage{algorithm}
\usepackage{algpseudocode}

\theoremstyle{plain}
\newtheorem{theorem}{Theorem}
\newtheorem{lemma}[theorem]{Lemma}

\newtheorem{fact}[theorem]{Fact}

\theoremstyle{definition}

\theoremstyle{remark}

\newcommand{\cK}{\mathcal{K}}
\newcommand{\ccK}{\cK(n, \mathbf{P})}
\newcommand{\cHH}{\mathcal{H}}
\newcommand{\tH}{\widehat{\mathcal{H}}}
\newcommand{\cH}{\overline{\mathcal{H}}}

\newcommand{\zzz}{\{0,1\}^n}
\DeclareMathOperator{\diam}{diam}

\newcommand{\eps}{\epsilon}
\newcommand{\pr}{\mathbb{P}}

\newcommand{\wN}{\widehat{N}}

\title{On the diameter of  Kronecker graphs}

\author[uam]{Tomasz {\L}uczak}
\ead{tomasz@amu.edu.pl}
\author[uam]{Justyna Tabor}
\ead{just.tabor@gmail.com}
\address[uam]{Faculty of Mathematics and Computer Science,
	 Adam Mickiewicz University, Pozna\'{n}, Poland}

\begin{document}

	\begin{abstract}
It is shown that a.a.s.\ as soon as a Kronecker graph becomes connected its diameter is bounded by a constant.

	\end{abstract}

	\maketitle



	\section{Introduction}
	A Kronecker graph is a random graph with vertex set $V=\zzz$, where the probability that two vertices $u, v\in V$ are adjacent strongly depends on the structure of the vectors $u=(u_{1},\ldots,u_{n})$, and $v=(v_{1},\ldots,v_{n})$. More specifically,
	let $\mathbf{P}$ be a symmetric matrix
	$$
	\mathbf{P}=\bordermatrix{&1&0\cr
		1&\alpha & \beta\cr
		0&\beta  & \gamma},
	$$
	where zeros and ones are labels of rows and columns of $\mathbf{P}$,
	$\alpha, \beta, \gamma\in [0,1]$, and $\alpha\ge \gamma$. In the Kronecker graph $\cK(n, \mathbf{P})$
	two vertices $u=(u_{1},\ldots,u_{n}), v=(v_{1},\ldots,v_{n})\in V=\zzz$ are
	adjacent with probability
	\begin{equation}
	\nonumber
	p_{u,v}=\prod\limits_{i=1}^n \mathbf{P}[u_{i},v_{i}],
	\end{equation}
	independently for each such pair.

	Kronecker graphs were introduced
	by Leskovec, Chakrabarti, Kleinberg and Faloutsos in \cite{Leskovec2005}
	to model some  real world networks  (see also \cite{Leskovec2010},
	\cite{Leskovec2007},  \cite{Kolda}). Since then they have been studied by
	several authors but their properties are still far from being well understood
	(see  \cite{FK} and references therein). In particular,
	Radcliffe and Young \cite{RadcliffeYoung}
	determined the exact threshold for the property that $\ccK$ is connected,
	supplementing a slightly weaker  result of Mahdian and Xu~\cite{MahdianXu}.

	\begin{theorem}\label{thm:con}
		$$
		\lim_{n\to\infty}\mathbb{P}(\ccK \mbox{ is connected})=\left\{
		\begin{array}{ll}
		0&\mbox{if }\beta+\gamma=1,\: \beta\neq1\\
		0&\mbox{if }\beta=1,\:\alpha=\gamma=0\\
		1&\mbox{if }\beta=1,\;\alpha>0\mbox{ and }\gamma=0\\
		1&\mbox{if }\beta+\gamma>1.
		\end{array}
		\right.
		$$
	\end{theorem}

	The main result of this work states that as soon as
	$\ccK$ becomes connected its diameter is bounded by a constant.

	\begin{theorem}\label{thm:diam}
		If either $\beta+\gamma>1$, or $\beta=1$, $\alpha>0$ and $\gamma=0$,
		then  there exists a constant $a=a(\alpha, \beta,\gamma)$ such that
		a.a.s. $\diam(\ccK)\le a$.
	\end{theorem}

	\section{The idea of the proof}

	In order to sketch our argument let us recall how one shows that
	the diameter is bounded from above for the binomial model of
	random graph $G(N,p)$, and for many other random graph models.
	Typically, since random
	graphs are good expanders,  it is proven first that for some small
	$k$ the $k$-neighbourhood of each vertex is
	much larger than $\sqrt N$. Then, in the second part of the proof, one argues
	that since two random subsets of vertices
	of size larger than $\sqrt N$  intersect with large probability,
	each pair of vertices is a.a.s. connected by a path of length at most $2k$.
	In \cite{MahdianXu}, the diameter of $\ccK$ was examined under condition $\gamma<\beta<\alpha$. This case was solved using the standard results for binomial random graphs.
However, for other cases  this procedure fails completely. The main reason is that
most neighbours of a given vertex $v$  have a similar structure, and so
	the events    `$x\sim v$' and `$y\sim v$' are strongly correlated.
	Thus, the $k$-neighbourhood of a given vertex is very far from being a random subset,
	which is crucial for the second step of the procedure. Even more importantly,
	we do not understand  expanding properties of $\ccK$  and it is hard to control
	how fast the $k$-neighbourhoods of  a vertex $\ccK$ grows, which in most
	of the other random graph models is quite easy to investigate.

	Thus, we  apply a different approach.
	We consider two vertices,  $v$ and $u$
	which are very similar to each other (more specifically, we choose both of them
	from the middle layer of the $n$-cube and assume that they are at  small
	Hamming distance from each other). Then we generate their neighbourhoods
	at the same time until, for some $k$, we observe that
	the $k$-neighbourhood of $v$ would not expand on an expected rate.
	This is because many, in fact most,  candidates for  $(k+1)$-neighbours of~$v$
	has already been placed in the $i$-neighbourhood of $v$ for some $i\le k$.
	However, the chance that a vertex $x$ is in the  $i$-neighbourhood of $v$ is roughly
	the same as the probability that $x$ is in the  $i$-neighbourhood of $u$ so, if most
	potential $(k+1)$-neighbours of $v$ are already in its $k$-th neighbourhood,
	many of them are also in the  $k$-th neighbourhood of $u$. Consequently,
	there is a path of length at most $2k$ joining $v$ and $u$.

The structure of the paper is the following. First  we  treat
 a special  case $\beta=1$. Then we present the
crucial part of our argument showing that the subgraph  induced in of $\ccK$ by its middle layer
has a.a.s.\ a small diameter.
Finally, we complete the proof showing
	that a.a.s. each vertex of $\ccK$ is connected to  the middle layer by a short path.

\section{Case $\beta=1$}
In this section we show that if $\beta=1$, $\alpha>0$, and $\gamma=0$,
then the diameter of $\ccK$ is a.a.s. bounded by a constant.
This set of parameters $\alpha,\beta,\gamma$ is somewhat special as it is the only case, when $\gamma+\beta= 1$ and still
 $\ccK$ is a.a.s. connected.

We introduce some notation, we shall use throughout the paper.
By ${d}(v,u)$ we denote a Hamming distance between two vertices $v$ and $u$ and
$w(v)$ stands for the weight of a vertex $v=(v_{1},\ldots,v_{n})$, i.e.\
the number of ones in its label
	$$
	w(v)=\sum\limits_{i=1}^nv_i.
	$$
For a vertex $v=(v_1,\ldots,v_n)$, we set $\bar v=(1-v_1,\ldots,1-v_n)$.

Now let us go back to the case $\beta=1$.
Note  that
$$
\mathbb{P}(v\sim \bar v)=\beta^n=1,
$$
and observe that either $v$ or $\bar v$ has weight at least $n/2$.
Thus, to show the assertion it is
enough to verify that a.a.s. there exists a path of bounded length between every pair of vertices in $R$ defined as
$$
R=\{v\in V:w(v)\ge n/2\}.
$$

Let $v\in R$ be a vertex of weight $n/2\le w(v)< n$.
We show that it is joined  to the vertex $(1,1,\dots,1)$
by a short path. To this end let $\eta\in (0,1)$ denote the largest solution of
the equality
$$
\alpha^{\eta}(\alpha^2+1)^{1/2}=1+\eta.
$$
Let $u\in R$ be such
that $w(u)=w(v)+t$ and $d(v,u)=t$, where $0<t<\eta n$.
Consider a vertex $x$ such that if $v_i=0$, then $x_i=1$, and $w(x)=n-w(v)+j$,
for some $j\in[w(v)-1]$.
Then
$$
\mathbb{P}(x\sim v,x\sim u)=\alpha^j\beta^{n-j}\alpha^{j+t}\beta^{n-j-t}=\alpha^{2j+t}.
$$
Hence the probability $\rho(v,u)$ that $v$ and $u$ have no common neighbours
in $\ccK$ is bounded from above by
\begin{align*}
\rho(v,u)&\le \prod\limits_{j=1}^{w(v)-1}\big(1-\alpha^{2j+t}\big)^{\binom {w(v)}{j}}
\le \exp\Big (- \sum\limits_{j=1}^{w(v)-1}{w(v)\choose j}\alpha^{2j+t}\Big)\\
&\le \exp\Big (-\alpha^t\big((\alpha^2+1)^{w(v)}-\alpha^{2w(v)}-1\big)\Big)\\
&\le\exp\left(-\frac{\left(\alpha^{\eta}(\alpha^2+1)^{1/2}\right)^n}{2}\right)=
\exp\left(-\frac12 (1+\eta)^n\right)
=o(|R|^2) \,,
\end{align*}
Thus a.a.s. each such pair $u,v\in R$ has a common neighbour. Therefore every vertex in $\ccK$ is a.a.s. connected to a vertex $(1,\ldots,1)$ by a path of length at most $1/2\eta$.
Consequently, a.a.s. between each pair of vertices of $\ccK$ there exists a path of
length  at most $1/\eta+2$, where, let us recall,
 $\eta$ is a constant which depends only on $\alpha$.

\section{The middle layer}\label{s:main}
In this section  we deal with the case when $\beta+\gamma>1$.
Let us recall that we always assume that $\alpha\ge \gamma$ however,
since if we decrease $\alpha $ the diameter can only increase, we may and
shall assume that  $\alpha=\gamma$.
Furthermore, for technical reasons, we assume also that $n$ is even; if this
is not the case one can split $\ccK$ into two disjoint random subgraphs for which
the underlying cube is of even dimension, use the result to infer
 that each of them is a.a.s.\  of
bounded diameter and finally verify that a.a.s.\ there exists at least one edge joining
two of them.

Denote by $\cHH$ a subgraph of $\ccK$ induced by the middle layer, i.e. by the set of vertices
$$
U=\left\{v\in V(\ccK):w(v)=n/2\right\}.
$$

In this section we show that the diameter of $\cHH$ is bounded and the following theorem holds.

\begin{theorem} For every  $\alpha=\gamma$ and $\beta+\gamma>1$,
there exists a constant   $c$, such that a.a.s.
\begin{equation*}
\diam(\cHH)<c\,.
\end{equation*}
	\label{thm:med}
\end{theorem}

In order to verify  Theorem~\ref{thm:med} we show that a.a.s.\ each pair
of vertices $v$ and $u$ such that $d(v,u)$ is small is connected in
average  by  at least $n^{2}/4$
edge-disjoint paths of bounded length. However, although
the expected number of short paths between $v$ and $u$ grows exponentially
it is hard to control directly the size of the largest family of
edge-disjoint such paths. Thus, we  use the following simple and  natural approach.
Each edge of $\cHH$ we label randomly and independently with one of $n^2$ labels, i.e.
we split $\cHH$ randomly into $n^2$ edge disjoint parts. Of course each of $n^2$ parts is
the same random object, which can be obtained by keeping edges of $\cHH$ with probability $n^{-2}$.
We denote it by $\cH$. The crucial part of our argument is
captured by the following lemma.

\begin{lemma} For every  $\alpha=\gamma$ and $\beta+\gamma>1$,
there exist constants $\eps>0$ and $\xi>1$   such that
for each pair  $v,u$ of vertices of $\cH$ such that $d(u,v)<\epsilon n$
the probability that $u$ and $v$ are connected in $\cH$ by  a path of length at most  $4\lceil\log_{\xi}2\rceil$ is at least $1/4$.
	\label{l:main}
\end{lemma}


\begin{proof}
Let $u,v$  be two vertices from $U$ such that $d(u,v)<\epsilon n$,
	where $\eps>0$ is to be chosen later on.
We show that $u$ and $v$ are connected by a short path in a subgraph $\tH$
of $\cH$ which is defined as follows.

Let $I$ be the set of those $i\in[n]$ for which $u_i\neq v_i$.
	Clearly $|I|=d(u,v)$ and, since $w(u)=w(v)=n/2$,  $|I|$ is even.
The vertex set of $\tH$  consists only of those vertices which
have precisely $|I|/2$ ones inside $I$ and, consequently, exactly
$(n-|I|)/2$ ones in $[n]\setminus I$. We denote it by $U_I$, i.e.
	$$U_I=\{x\in U:\lvert\{ i\in I: x_i=1\}\rvert=\lvert\{ i\in I:x_i=0\}\rvert=|I|/2\}.$$
	Note that clearly $u,v \in U_I$.
	
Now we delete from $\cH$ some edges. So firstly, we leave in $\tH$ only those
edges which fulfill the following condition:
\begin{equation}\label{alpha,beta-neighbour1}
		\begin{aligned}
	&\lvert\{i:x_i=y_i\}\setminus I\rvert=\frac{\alpha}{\alpha+\beta}(n-|I|)\;.
	\end{aligned}
\end{equation}		
Furthermore, when computing the probability that vertices $x,y$ for which
the above holds are adjacent we would like to `ignore' the part of $x$ and $y$ which belong to
$I$. Note that the probability that such a given pair is connected by an edge in
$\cH$ is always at least
\begin{equation}\label{eq:rho}
\rho=\rho(\alpha,\beta)=\alpha^{\frac{\alpha}{\alpha+\beta}(n-|I|)}
		\beta^{\frac{\beta}{\alpha+\beta}(n-|I|)}(\min\{\alpha,\beta\})^{|I|}n^{-2}.
\end{equation}
Now,  each edge of $\cH$ for which the probability of existence $\rho'$ is larger
than~$\rho$ above we delete  with probability
$\rho/\rho'$. The edges which remain after these procedures are the edges of $\tH$.

Note that for every vertices $x$ and $y$ of $\tH$ which differ only on the set $I$
(such as $u$ and $v$) and
all other $z\in U_I$ the probability that $z$ is adjacent to $x$ is precisely the same
as the probability that $z$ is adjacent to $y$. Furthermore, all vertices
of $\tH$ are, in a way, equivalent. More precisely,
for every $x, y\in U_I$ a natural permutation of coordinates of $x$
induces a bijection of $U_I$
onto itself which transforms $x$ into $y$ and which preserves measure, i.e.
which transform~$\tH$ onto itself.

We shall show that degrees of vertices of  $\tH$  increase exponentially with~$n$.
Let us first make the following  useful observation. 	
It is easy to see (cf.~\cite{MahdianXu}) that
	 the expected size of neighbourhood of a vertex $v$ in $\ccK$ is equal
	 to
\begin{equation*}\label{eq:exp1}
\sum_{r=0}^{w(v)}\sum_{s=0}^{n-w(v)}\binom {w(v)}r\binom {n-w(v)}s\alpha^r\beta^{w(v)-r}\alpha^s
\beta^{n-w(v)-s}=(\alpha+\beta)^n.
\end{equation*}
Thus, the largest term in the sums on the left hand size is at least as large as
$(\alpha+\beta)^n/n^2$. Since we shall often refer to  this fact let us state it
explicitly. For every natural numbers $n,w$ and positive constants $\alpha, \beta<1$
we have
\begin{equation}\label{eq:exp}
\binom {w}{\frac{\alpha}{\alpha+\beta}w}
\binom {n-w}{\frac{\alpha}{\alpha+\beta}(n-w)}
\alpha^{\frac{\alpha}{\alpha+\beta}n}\beta^{\frac{\alpha}{\alpha+\beta}n}
\ge	(\alpha+\beta)^n/n^2.
\end{equation}

Using the above inequality we can easily estimate degrees of vertices in~$\tH$.
Let us denote by 	$\widehat{N}(x)$  the set of all  neighbours of~$x$ in $\tH$, and,
more generally,
by $\widehat{N}_i(x)$ we denote the $i$-th neighbourhood of $x$ (i.e.\ the set of all
$y$ which lies at the distance $i$ from $x$ in $\tH$). Then the following holds.

	\begin{fact}
Let $\eps>0$ be a small constant and let $I$ be  a given subset of $[n]$ such that
$|I|\le \eps n$. Then, if $\eps>0$ is small enough,
there exists a  constant $\xi>1$ such that 
for every $x\in U_I$ we have
$$\pr\left(|\widehat{N}(x)|\ge \xi^n\right)=1-o(1)\,.   $$
		\label{fakt:expDobreSasiedztwo}
	\end{fact}
	\begin{proof} For a given $x\in U_I$ the random variable $\widehat{N}(x)$
	has the binomial distribution $B(M,\rho)$, where
$$M=	{(n-|I|)/2\choose \frac{\alpha}{\alpha+\beta}(n-|I|)/2}^2,$$
and $\rho$ is given by \eqref{eq:rho}.
From (\ref{eq:exp}) and the fact that $|I|<\epsilon n$ we get
		\begin{align*}
\mathbb{E}\widehat{N}(x)&
		=(\min\{\alpha,\beta\})^{|I|}{(n-|I|)/2\choose \frac{\alpha}{\alpha+\beta}(n-|I|)/2}^2
		\alpha^{\frac{\alpha}{\alpha+\beta}(n-|I|)}
		\beta^{\frac{\beta}{\alpha+\beta}(n-|I|)}n^{-2}\\
		&
		\ge(\min\{\alpha,\beta\})^{\epsilon n}(\alpha+\beta)^{(1-\epsilon)n}/n^4.
		\end{align*}
Hence, if $\eps$ is small enough, for some constant $\xi>1$ we have
$\mathbb{E}\widehat{N}(x)\ge 2\xi^n$ and the assertion follows from
either Chernoff's or
Chebyshev's inequalities.
	\end{proof}
	
Now let $k=2\lceil \log_{\xi}2\rceil $. Recall that we need to show that
\begin{equation}\label{eq21}
\pr\left(\widehat{N}_k(v)\cap \widehat{N}_k(u)\neq\emptyset\right)\ge 1/4\,.
\end{equation}
In fact we shall show a slightly stronger inequality. Let us first split the
set $\wN(v)$ randomly into two sets $\wN_{(H)}(v)$ and $\wN_{(T)}(v)$ by tossing for
each vertex of $\wN(v)$ a symmetric coin. In the same way we
partition $\wN(u)$ into sets $\wN_{(H)}(u)$ and $\wN_{(T)}(u)$. Furthermore, let $\wN_k^{-y}(x)$ denote
the $k$-neighbourhood of $x$ in $\tH$ from which we removed an edge $\{x,y\}$ if such an edge exists. We shall show that
\begin{equation}\label{eq22}
\pr\left(\bigcup_{x\in \wN_{(T)}(v)}\widehat{N}_{k-1}^{-v}(x)\cap
\bigcup_{y\in \wN_{(H)}(u)}\widehat{N}^{-u}_{k-1}(y)\neq\emptyset\right)\ge 1/4.
\end{equation}
	
Let us assume that \eqref{eq22} does not hold, i.e. that we have
	
	\begin{equation}\label{eq23}
\pr\left(\bigcup_{x\in \wN_{(T)}(v)}\widehat{N}^{-v}_{k-1}(x)\cap
\bigcup_{y\in \wN_{(H)}(u)}\widehat{N}^{-u}_{k-1}(y)\neq\emptyset\right)< 1/4.
\end{equation}
			
We shall show that \eqref{eq23} leads to a contradiction.

			Let us start with a simple but crucial observation.
In order to check if the event
$$	\bigcup_{x\in \wN_{(T)}(v)}\widehat{N}^{-v}_{k-1}(x)\cap
\bigcup_{y\in \wN_{(H)}(u)}\widehat{N}^{-u}_{k-1}(y)\neq\emptyset$$
holds we need first to generate the set  $\bigcup_{x\in \wN_{(T)}(v)}\widehat{N}^{-v}_{k-1}(x)$
and then generate vertices of $\bigcup_{y\in \wN_{(H)}(u)}\widehat{N}^{-u}_{k-1}(y)$
until the moment when we meet the first vertex which belongs to both of these sets.
But the distribution of $\bigcup_{y\in \wN_{(H)}(u)}\widehat{N}^{-u}_{k-1}(y)$
is identical with $\bigcup_{y\in \wN_{(H)}(v)}\widehat{N}^{-v}_{k-1}(y)$. Thus,
\begin{multline*}
\pr\left(\bigcup_{x\in \wN_{(T)}(v)}\widehat{N}^{-v}_{k-1}(x)\cap
\bigcup_{y\in \wN_{(H)}(u)}\widehat{N}^{-u}_{k-1}(y)\neq\emptyset\right)\\
\ge
\pr\left(\bigcup_{x\in \wN_{(T)}(v)}\widehat{N}^{-v}_{k-1}(x)\cap
\bigcup_{y\in \wN_{(H)}(v)}\widehat{N}^{-v}_{k-1}(y)\neq\emptyset\right)
\end{multline*}
where we have inequality instead of equality because, perhaps, $\wN_{(T)}(v)\cap
\wN_{(H)}(u)\neq \emptyset$.			
			Hence, from \eqref{eq23} it follows that
$$\pr\left(\bigcup_{x\in \wN_{(T)}(v)}\widehat{N}^{-v}_{k-1}(x)\cap
\bigcup_{y\in \wN_{(H)}(v)}\widehat{N}^{-v}_{k-1}(y)\neq\emptyset\right)	<1/4,
$$
and, since we may toss a coin  after we generate all vertices
from $\widehat{N}(v)$,
	\begin{equation}\label{eq24}
\pr\left(\exists_{x,y\in \wN(v), x\neq y}\colon \widehat{N}^{-v}_{k-1}(x)\cap
\widehat{N}^{-v}_{k-1}(y)\neq\emptyset\right)< 1/2.
\end{equation}
The above inequality seems to lead to contradiction at very first sight.
Indeed, since the degree of vertices of $\tH$ grows like $\xi^n$, after
finite number of steps we explore all the vertices of $\tH$ and simply
 there will be no place for new vertices. However, keeping in
mind that the neighbours of a vertex of $\tH$ are strongly correlated, we have
to exclude the possibility that each $\wN^{-v}_{k-1}(x)$ spans a dense graph yet
for different $x$ and $y$ the sets  $\widehat{N}^{-v}_{k-1}(x)$ and
$\widehat{N}^{-v}_{k-1}(y)$  are disjoint.

Thus, let $J$ be a random variable, defined as
$$J=\min\{i: |\wN_i(v)|\le \xi^{ni/2}\}\,.$$
Note that since for $k=2\lceil 	\log_{\xi}2\rceil $
we have $\xi^{nk/2}>2^n$, $J$ takes only finite number of
values, more precisely, we have $J\in[k]$. Let $j\in [k]$
maximize the value of $\pr(J=j)$.
Then, clearly,

\begin{equation}\label{eq25}
\pr\left(|\wN_{j-1}(v)|\ge \xi^{n(j-1)/2}\ \&\ |\wN_{j}(v)|\le \xi^{nj/2}\right)
\ge \pr(J=j)\ge 1/k\,.
\end{equation}

Since with large probability we have  $|\wN(v)|\ge \xi^n$, and due to \eqref{eq25}
for each $x\in \wN(v)$ we have $|\wN_{j-1}(x)|\sim  |\wN_{j-1}(v)|\ge \xi^{n(j-1)/2}$,
then, clearly, if most of $\wN_{j-1}(x)$, $\wN_{j-1}(y)$ are disjoint (see \eqref{eq23})
we must also
have $|\wN_j(v)|\ge \xi^{n(j+1)/2}$ contradicting \eqref{eq25}.

Let us make the above heuristic argument rigorous. Recall that with probability $1-o(1)$
we have $|\wN(v)|\ge \xi ^n$ (Fact~\ref{fakt:expDobreSasiedztwo}). Furthermore,
since for every $x\in \wN(v)$ the random variables  $|\wN_{j-1}(x)|$ and $|\wN_{j-1}(v)|$
have identical distribution, and the degree of $x$ is binomially distributed with
exponential expectation, from \eqref{eq25} it follows that for every
 $x\in \wN_1(v)$ we have
 \begin{equation}\label{eq26}
\pr\left(|\wN_{j-1}^{-v}(x)|\ge \xi^{n(j-1)/2}/2\right)
\ge 1/(2k)\,.
\end{equation}
Let us generate all neighbours of $v$ and order them in a sequence
$x_1,x_2,\dots, x_r$, where, as we have just observed, with probability $1-o(1)$
we have $r\ge \xi^n$. Then,  for each $x_i$ we generate a set $W_i$
in the following way. We set  $W_1=\wN^{-v}_{j-1}(x_1)$.
Once all sets $W_1,\dots, W_{i-1}$ are found, we generate vertices of
the $(j-1)$-neighbourhoods of $x_i$
one by one and stop when we either find the whole neighbourhood $\wN^{-v}_{j-1}(x_i)$,
or when we first hit an element of $\bigcup_{s=1}^{i-1}W_s$. In this latter case
we set  as $W_i$ the set of all vertices of  $\wN^{-v}_{j-1}(x_i)$ generated so far
and mark $x_i$ as {\em bad}.

\begin{fact}\label{bad}
With probability $1-o(1)$ none of the vertices $x_1,x_2,\dots, x_t$, where $t=\xi^{2n/3}$,
is bad.
\end{fact}

\begin{proof} During the execution of the procedure of generating $W_1,W_2,\dots, W_r$, the probabilities
$\theta_i$ that $W_i$ is bad increases with $i$.
Thus, for every $i\in [t]$, $\theta_i\le \theta_{t+1}$. On the other hand
the probability that there are bad vertices among $x_{t+1},x_{t+2}, \dots, x_r$ is
bounded from below by $1-(1-\theta_{t+1})^{r-t}$ while, from our assumption \eqref{eq24},
it is bounded from above by $1/2$.
Hence,
$$1-(1-\theta_{t+1})^{r-t}<1/2\,,$$
and so, very crudely, $\theta_{t+1}\le \xi^{-3n/4}$.  Thus, the probability that
one of vertices $x_1,\dots, x_t$ is bad is bounded from above by
$$\sum_{i=1}^t\theta_i\le  t\theta_{t+1}\le \xi^{2n/3}\xi^{-3n/4}=o(1).\quad\qed$$
\renewcommand{\qed}{}
\end{proof}
Since with probability $1-o(1)$  sets $W_i$,  $i=1,2,\dots,t$, 
are disjoint and thus equal to $\wN_{j-1}^{-v}(x_i)$, and since  
each of such sets with probability $1/(2k)$ have size at least $\xi^{(j-1)n/2}/2$,
with probability $1-o(1)$ we have
$$|\wN_{j}(v)|\ge  \left\lvert\bigcup_{i=1}^t \wN_{j-1}^{-v}(x_i)\right\rvert/(3k)\ge t \xi^{(j-1)n/2}/(6k)>
\xi^{jn/2}$$
contradicting \eqref{eq25}. Thus, \eqref{eq23} leads to a contradiction and
Lemma~\ref{l:main} holds.
\end{proof}

\begin{proof}[Proof of Theorem~\ref{thm:med}]
Let $\epsilon$ and $\xi$ be constants for which the assertion of
Lemma~\ref{l:main} holds. Consider two vertices $v,u$ of $\cHH$ such that
$d(u,v)<\epsilon n$ and let $X_{v,u}$ denote the maximum number of
edge-disjoint path of length at most $4\lceil\log _\xi 2\rceil$ joining $v$ and $u$ in $\cHH$.
From Lemma~\ref{l:main} it follows that $\mathbb{E} X_{v,u}\ge n^2/4$.
 Since adding or removing a single  edge cannot affect the value of $X_{v,u}$
by more than one, and the fact  that $X_{v,u}\ge m$ can be verified by  exposing only
$4m\lceil\log_{\xi}2 \rceil$ edges, from Talagrand's inequality (see, for instance,~\cite{JLR}, Theorem~2.29) we get that for some constant $\eta=\eta(\eps,\xi)>0$
$$\mathbb{P}(X_{v,u}< n^2/8)\le \exp(-\eta n^2)\,.$$
Thus, by the first moment method a.a.s. each pair of vertices $v,u$, of $\cHH$
such that
$d(v,u)\le \eps n$ is connected by a path of length at most $4\lceil \log_{\xi}2 \rceil$.
To complete the proof it is enough to observe that
 for every pair of vertices $u,u'$ of $\cHH$  one can find a sequence
$$
u=u_0, u_1,\ldots, u_r=u'
$$
such that for $i\in[r]$, $d(u_i,u_{i-1})<\epsilon n$, and $r<2/\epsilon$.
Consequently,
$$
\diam(\cHH)\le  8\lceil \log_{\xi}2/\epsilon\rceil .\quad \qed
$$
\renewcommand{\qed}{\ }
\end{proof}

\section{A final touch: reaching the middle layer}
In this section we prove that a.a.s. every vertex of $\ccK$ is joined
to the middle layer by a short path, i.e. the following result holds.
\begin{theorem}
For every $\alpha=\gamma$ and $\beta+\gamma>1$ there exists $c'=c'(\alpha, \beta)$ such that
a.a.s. every vertex $v$ of $\ccK$ is connected by a path of length at most~$c'$
to the middle layer of $\ccK$.
\label{thm:notmid}
\end{theorem}
Theorem~\ref{thm:notmid} is a direct consequence of the following three lemmata.

\begin{lemma}\label{l:m1}
	Let $\alpha=\gamma$ and $\beta+\gamma>1$. Then a.a.s.
each vertex  $v$  of $\ccK$ of weight $w(v)\neq n/2$ has a neighbour of weight
	$$
\frac{n}{2}+\frac{\alpha-\beta}{\alpha+\beta}\left(w(v)-\frac{n}{2}\right).
	$$
\end{lemma}
\begin{proof} For a vertex $v$  consider the set of all vertices of $\ccK$ which have respectively $\frac{\alpha}{\alpha+\beta}w(v)$ and $\frac{\alpha}{\alpha+\beta}(n-w(v))$ common ones and common zeros with $v$. There are precisely
	$$
	{w(v)\choose \frac{\alpha}{\alpha+\beta}w(v)}{n-w(v)\choose\frac{\alpha}{\alpha+\beta}(n-w(v))}
	$$
of them. Each of these vertices has weight
	$$
	\frac{\alpha}{\alpha+\beta}w(v)+\frac{\beta}{\alpha+\beta}(n-w(v))=\frac{n}{2}+\frac{\alpha-\beta}{\alpha+\beta}\left(w-\frac{n}{2}\right).
	$$
	Moreover, each of them is adjacent to $v$ with probability
	$$
	\alpha^{\frac{\alpha}{\alpha+\beta}n}\beta^{\frac{\beta}{\alpha+\beta}n}.
	$$
	Let $X(v)$ be a random variable which counts those neighbours of $v$. Then, by (\ref{eq:exp}),
	we get
	$$
	\mathbb{E}X(v)={w(v)\choose \frac{\alpha}{\alpha+\beta}w(v)}{n-w(v)\choose\frac{\alpha}{\alpha+\beta}(n-w(v))}\alpha^{\frac{\alpha}{\alpha+\beta}n}\beta^{\frac{\beta}{\alpha+\beta}n}\ge\frac{(\alpha+\beta)^n}{n^2}.
	$$
	By Chernoff bound, with probability at least
	$$
	1-\exp(-\mathbb{E}X(v)/8)\ge 1-\exp(-n^2),
	$$
	$v$ has more than $\frac{1}{2}\mathbb{E}X(v)$ neighbours of weight
	$$
	\frac{n}{2}+\frac{\alpha-\beta}{\alpha+\beta}\left(w-\frac{n}{2}\right).
	$$
Hence, by the first moment method, the probability that $\ccK$ contains a vertex
which does not have this property tends to 0 as $n\to\infty$.
\end{proof}
\begin{lemma}\label{l:m2}
 	Let $\alpha=\gamma$ and $\beta+\gamma>1$, $\eps>0$, and let
	$$
	b=\log_{\left\lvert\frac{\alpha-\beta}{\alpha+\beta}\right\rvert}(\epsilon).
	$$
Then a.a.s. every vertex $v$ of $\ccK$ such that $|w(v)-n/2|>\epsilon n$
is connected by a path  of length at most $b$ to a vertex $u$ with weight $w(u)$ such that
	$$
	\Big|w(u)-\frac{n}{2}\Big|\le\epsilon\left(w(v)-\frac{n}{2}\right).
	$$
\end{lemma}
\begin{proof}
By Lemma~\ref{l:m1}, a.a.s. for each vertex $v$ there exists a path
	$$
	v=v_0\sim v_1\sim v_2\sim\ldots v_b
	$$
	such that
	$$w(v_i)=\frac{n}{2}+\frac{\alpha-\beta}{\alpha+\beta}\left(w(v_{i-1})-\frac{n}{2}\right),$$
	for $i\in[b]$.
	Solving the above recurrence we get
	$$
	w(v_b)=\frac{n}{2}+\left(\frac{\alpha-\beta}{\alpha+\beta}\right)^b\left(w(v)-\frac{n}{2}\right),
	$$
and so the assertion follows.
\end{proof}

\begin{lemma}\label{l:m3}
Let $\alpha=\gamma$ and $\beta+\gamma>1$. Moreover, let $\epsilon>0$ be such that
	\begin{equation}
	(\alpha+\beta)^{1-\epsilon}(4\alpha\beta)^{\epsilon}>1.
	\label{epsilon}
	\end{equation}
Then a.a.s. every vertex $v$ of $\ccK$ of weight $w(v)$ such that $|w(v)-n/2|\le \epsilon n/2$
has a neighbour with weight $n/2$.
\end{lemma}
\begin{proof}
For a given vertex $v$, let 	$|w(v)-n/2|=\eta n/2$, where $\eta\le\epsilon$. Let us choose $(1-\eta)n/2$ ones and $(1-\eta)n/2$ zeros in the label of $v$. Let $A$ denote the set of
those vertices, which have precisely $\frac{\alpha}{\alpha+\beta}(1-\eta)n/2$ ones among the chosen one-positions in the label of $v$,  a $\frac{\alpha}{\alpha+\beta}(1-\eta)n/2$ zeros among the chosen zero-positions in the label of $v$, and have $\eta n/2$ zeros among remaining $\eta n$ positions. Then
	$$
	|A|={(1-\eta)n/2\choose \frac{\alpha}{\alpha+\beta}(1-\eta)n/2}^2{\eta n \choose \frac{1}{2}\eta n}.
	$$
	Furthermore, for every $u\in A$, the probability that $u$ is a neighbour of $v$ is
	$$
	\alpha^{\frac{\alpha}{\alpha+\beta}(1-\eta)n}\beta^{\frac{\beta}{\alpha+\beta}(1-\eta)n}\alpha^{\frac{1}{2}\eta}\beta^{\frac{1}{2}\eta}.
	$$
	Let $X(v)$ be a random variable which counts neighbours of $v$ in $A$. Then, by (\ref{eq:exp}),
	we have
	\begin{equation*}
	\begin{split}
	\mathbb{E}X(v)&={(1-\eta)n/2\choose \frac{\alpha}{\alpha+\beta}(1-\eta)n/2}^2{\eta n \choose \frac{1}{2}\eta n}\alpha^{\frac{\alpha}{\alpha+\beta}(1-\eta)n}\beta^{\frac{\beta}{\alpha+\beta}(1-\eta)n}\alpha^{\frac{1}{2}\eta}\beta^{\frac{1}{2}\eta}\\&
	\ge
	\left((\alpha+\beta)^{1-\eta}(4\alpha\beta)^{\eta}\right)^n/n^2.
	\end{split}
	\end{equation*}
	As $\eta\le\epsilon$, from \eqref{epsilon} we get
	$$
	\mathbb{E}X(v)\ge\zeta^n,
	$$
for some $\zeta>1$. Moreover for every $u\in A$,
	$$
	w(u)=\frac{\alpha}{\alpha+\beta}(1-\eta)n/2+\frac{\beta}{\alpha+\beta}(1-\eta)n/2+\eta n/2=n/2.
	$$
	By Chernoff's inequality, with probability at least $1-\exp(-\mathbb{E}X(v)/8)$, $v$ has a neighbour with weight $n/2$. Thus, the probability that some vertex $v$ of weight $w(v)$ such that
	$$
	|w(v)-n/2|\le \epsilon n/2\,,
	$$
has no neighbours of weight $n/2$ tends to 0 as $n\to\infty $.
\end{proof}

\begin{proof}[Proof of Theorem~\ref{thm:notmid}] Theorem~\ref{thm:notmid} is a
straightforward consequence of Lemmata~\ref{l:m1}, \ref{l:m2}, and \ref{l:m3}.
\end{proof}

Now it is easy to complete the proof of Theorem~\ref{thm:diam}.

\begin{proof}[Proof of Theorem~\ref{thm:diam}] From Theorem~\ref{thm:notmid}
we know that for some constant $c'$ a.a.s. from every vertex $v$ of $\ccK$ we can reach the middle layer in  at most $c'$ steps. Moreover, Theorem~\ref{thm:med} states that for some constant $c$
 a.a.s.  the
diameter of the subgraph of $\ccK$ induced by the middle layer is at most $c$.
Consequently, the diameter of $\ccK$ is a.a.s. bounded from above by $c+2c'$.
 \end{proof}

\bigskip

{\bf Acknowledgment}: The authors were partially supported   by Maestro NCN grant 2012/06/A/ST1/00261.



\begin{thebibliography}{99}
	
	\bibitem{Leskovec2005} D.~Chakrabarti, C.~Faloutsos, J.~Kleinberg and J.~Leskovec. \newblock Realistic, mathematically tractable graph generation and evolution.  \newblock {\em Proceedings of the European Conference on Principles and Practice of Knowledge Discovery in Databases}, 133--145 2005.
	
	
	\bibitem{Leskovec2010} D.~Chakrabarti, C.~Faloutsos, Z.~Ghahramani, J.~Kleinberg and J.~Leskovec.  \newblock Kronecker graphs: an approach to modeling networks.  \newblock {\em Journal of Machine Learning Research},
	11: 985--1042, 2010.
	
	
	\bibitem{Leskovec2007} C.~Faloutsos and J.~Leskovec.  \newblock Scalable modeling of real graphs using Kronecker multiplication.  \newblock {\em International Conference on Machine Learning (ICML)}, 497--504, 2007.
	
		\bibitem{Kolda} T.~Kolda, A.~Pinar and C.~Seshadhri.  \newblock An in-depth analysis of stochastic Kronecker graphs.  \newblock {\em Journal of the ACM},
		60(2), 2013.
		
	\bibitem{FK} A.M.~Frieze and M.~Karo\'nski. \newblock {\em Introduction to Random Graphs}. \newblock Cambridge University Press, 2015.
	\bibitem{RadcliffeYoung} M.~Radcliffe and S.~Young.
	\newblock	Connectivity and giant component of stochastic Kronecker graphs.
	\newblock {\em  Journal of Combinatorics}, 6(4):457-482, 2015. 
		\bibitem{MahdianXu} M.~Mahdian and Y.~Xu. \newblock Stochastic Kronecker graphs.  \newblock {\em Random Structures and Algorithms}, 38(4):453--466, 2011.
	\bibitem{JLR} S.~Janson, T.~{\L}uczak,  A.~Ruci\'nski. \newblock
	{\em Random Graphs}. \newblock Wiley, 2000.
	
\end{thebibliography}
\end{document}